\documentclass[reqno]{amsart}

\usepackage{hyperref}
\usepackage[all]{xy}
\usepackage{mathrsfs}
\usepackage{amssymb}

\numberwithin{equation}{section}

\newtheorem{mainthm}[equation]{Theorem}

\newtheorem{thm}[equation]{Theorem}
\newtheorem{cor}[equation]{Corollary}
\newtheorem{lem}[equation]{Lemma}
\newtheorem{prop}[equation]{Proposition}

\theoremstyle{remark}

\newtheorem{rem}[equation]{Remark}

\theoremstyle{definition}

\newtheorem{defn}[equation]{Definition}

\newcommand{\bQ}{\mathbb{Q}}
\newcommand{\bR}{\mathbb{R}}

\newcommand{\bZ}{\mathbb{Z}}
\newcommand{\bC}{\mathbb{C}}

\newcommand{\bD}{\mathbb{E}}

\newcommand{\proj}{\mathrm{pr}}
\newcommand{\Diff}{\mathrm{Diff}}
\newcommand{\map}{\mathrm{map}}

\newcommand{\smb}{\mathrm{smb}}

\newcommand{\bott}{\mathrm{bott}}
\newcommand{\even}{\mathrm{even}}
\newcommand{\odd}{\mathrm{odd}}

\newcommand{\rank}{\mathrm{rank}}

\newcommand{\SU}{\mathrm{SU}}
\newcommand{\Un}{\mathrm{U}}
\newcommand{\SO}{\mathrm{SO}}

\newcommand{\id}{\mathrm{id}}
\newcommand{\midex}{\mu}

\newcommand{\cL}{\mathcal{L}}
\newcommand{\cA}{\mathcal{A}}

\newcommand{\Eig}{\mathrm{Eig}}
\newcommand{\End}{\mathrm{End}}

\newcommand{\scpr}[1]{\langle #1 \rangle}
\newcommand{\bN}{\mathbb{N}}

\newcommand{\MT}{\mathrm{MT}}
\newcommand{\MTSO}{\mathrm{MTSO}}
\newcommand{\hocolim}{\mathrm{hocolim}}
\newcommand{\sch}{\mathrm{sch}}
\newcommand{\ch}{\mathrm{ch}}
\newcommand{\vol}{\mathrm{vol}}

\newcommand{\Cl}{\mathrm{Cl}}

\newcommand{\ind}{\mathrm{ind}}

\newcommand{\Tr}{\mathrm{Tr}}

\newcommand{\bP}{\mathbb{P}}

\newcommand{\norm}[1]{\|#1\|}
\newcommand{\Spin}{\mathrm{Spin}}
\newcommand{\twomatrix}[4]{\begin{pmatrix} #1 & #2 \\ #3& #4 \end{pmatrix}}

\newcommand{\ext}{\mathrm{ext}}

\newcommand{\grotimes}{\hat{\otimes}}
\newcommand{\grboxtimes}{\hat{\boxtimes}}

\newcommand{\sign}{\mathrm{sign}}

\newcommand{\gb}{\mathbf{b}}

\title{Tautological classes and higher signatures}

\author{Johannes Ebert}
\email{johannes.ebert@uni-muenster.de}
\address{
Mathematisches Institut\\
WWU M{\"u}nster\\
Einsteinstr. 62\\
48149 M{\"u}nster\\
Germany
}

\thanks{The author was supported by the Deutsche Forschungsgemeinschaft (DFG, German Research Foundation) -- Project-ID 427320536 -- SFB 1442, as well as under Germany’s Excellence Strategy EXC 2044 -- 390685587, Mathematics M\"unster: Dynamics–Geometry–Structure.
}

\date{\today}

\begin{document}

\begin{abstract}
For a bundle of oriented closed smooth $n$-manifolds $\pi: E \to X$, the tautological class $\kappa_{\cL_k} (E) \in H^{4k-n}(X;\bQ)$ is defined by fibre integration of the Hirzebruch class $\cL_k (T_v E)$ of the vertical tangent bundle. More generally, given a discrete group $G$, a class $u \in H^p(B G;\bQ)$ and a map $f:E \to B G$, one has tautological classes $\kappa_{\cL_k ,u}(E,f) \in H^{4k+p-n}(X;\bQ)$ associated to the Novikov higher signatures. 

For odd $n$, it is well-known that $\kappa_{\cL_k}(E)=0$ for all bundles with $n$-dimensional fibres. The aim of this note is to show that the question whether more generally $\kappa_{\cL_k,u}(E,f)=0$ (for odd $n$) depends sensitively on the group $G$ and the class $u$. 

For example, given a nonzero cohomology class $u \in H^2 (B \pi_1 (\Sigma_g);\bQ)$ of a surface group, we show that always $\kappa_{\cL_k,u}(E,f)=0$ if $g \geq 2$, whereas sometimes $\kappa_{\cL_k,u}(E,f)\neq 0$ if $g=1$. 

The vanishing theorem is obtained by a generalization of the index-theoretic proof that $\kappa_{\cL_k}(E)=0$, while the nontriviality theorem follows with little effort from the work of Galatius and Randal-Williams on diffeomorphism groups of even-dimensional manifolds.
\end{abstract}

\maketitle

\tableofcontents

\section{Introduction}

For a bundle $\pi:E \to X$ of smooth oriented closed $n$-manifolds and a characteristic class $c \in H^k (B\SO(n); \bQ)$, the tautological classes are defined as 
\[
\kappa_c (E):= \pi_! (c(T_v E)) \in H^{k-n}(X;\bQ),
\] 
where $T_v E \to E$ is the vertical tangent bundle and $\pi_!$ the Gysin homomorphism in cohomology. If $M$ is the fibre of $\pi$, this construction gives in the universal case a class in 
\[
H^{k-n}(B \Diff^+ (M);\bQ). 
\]

More generally, we can consider a further space $B$ and oriented closed manifold bundles $\pi:E \to X$ together with maps $f:E \to B$; given $u \in H^p (B;\bQ)$, we put
\[
\kappa_{c,u}(E,f):= \pi_! (c(T_v E) \cup f^* u) \in H^{k+p-n}(X;\bQ). 
\]
The universal example of such a bundle has $E \Diff^+(M)\times_{\Diff^+(M)} \map (M;B)$ as base space. 

We are interested in the special case of the above where $B=BG$ is the classifying space of a discrete group and $c= \cL_k \in H^{4k}(BO;\bQ)$ is the $k$th component of the Hirzebruch $L$-class, or rather Atiyah--Singer's variant \cite{AtiyahSingerIII} thereof: 
\[
\cL \in \hat{H}^{4*}(BO;\bQ) :=\prod_{k \geq 0} H^{4k}(BO;\bQ)
\]
is the multiplicative characteristic class associated to the power series $\frac{x/2}{\tanh(x/2)} \in \bQ[[x^2]]$ (Hirzebruch's original class $L$ is associated to $\frac{x}{\tanh(x)}$ and differs in each degree from $\cL$ by a power of $2$). 

In that case, the class $\kappa_{\cL_k,u}(E,f) \in H^{4k+p-n}(X;\bQ)$ is a family version of the Novikov higher signature in the sense that when $4k+p=n$, the class $\kappa_{\cL_k,u}(E,f) \in H^0 (X;\bQ)$ evaluated at $x \in X$ is Novikov's higher signature 
\[
\sign_u (E_x,f|_{E_x}):= \scpr{\cL_k (TE_x) \cup (f|_{E_x})^* u,[E_x]} \in \bQ
\]
of the fibre $E_x$ over $x$. 

It is well-known that $\kappa_{\cL_k}(E)=0$ for bundle with odd fibre dimension; the author has given an index-theoretic proof of this fact in \cite{Ebert13}. One might ask whether $\kappa_{L_k,u} (E,f)=0$ holds more generally for bundles with odd fibre dimension, $u  \in H^* (BG;\bQ)$ and maps $f:E \to BG$. The aim of this note is to disclose that the answer depends sensitively on $G$ and $u$; in a way that seems to be surprising and puzzling to the author. More specifically, we shall prove the following result; we believe it to be of interest because of the contrast between statements (1) and (2) on one side and (3) on the other.

\begin{mainthm}\label{mainthm:surfacegroup}
Let $\Gamma_g:= \pi_1 (\Sigma_g)$ be the fundamental group of a surface of genus $g \geq 1$ and let $w \in H^1 (B\Gamma_g;\bQ)$ and $u \in H^2 (B\Gamma_g;\bQ)$ be nonzero.
\begin{enumerate}
\item If $g \geq 1$, $m \geq 1$ and $4k \geq 2m$, there is some bundle of closed oriented $(2m+1)$-manifolds $E \to X$ and a map $f: E \to B \Gamma_g$ such that $\kappa_{\cL_k,w}(E,f)\neq 0 \in H^{4k-2m}(X;\bQ)$. 
\item If $g=1$, $m \geq 1$ and $4k \geq 2m$, there is some bundle of closed oriented $(2m+1)$-manifolds $E \to X$ and a map $f:E \to B \Gamma_1 \simeq T^2$ such that $\kappa_{\cL_k,u}(E,f) \neq 0 \in H^{4k+1-2m}(X;\bQ)$.
\item If $g \geq 2$ and $n$ is odd, then $\kappa_{\cL_k,u}(E,f) = 0 \in H^{4k+2-n}(X;\bQ)$ for each bundle of closed oriented $n$-dimensional manifolds $E \to X$ and each $f:E \to B \Gamma_g$.
\end{enumerate}
\end{mainthm}
The three statements are special cases of more general results that we now state. The general version of Theorem \ref{mainthm:surfacegroup} (1) and (2) reads as follows. 

\begin{mainthm}\label{thm:nontriviality}
Let $G$ and $H$ be groups, $v \in H^q (BG;Q)$, $w \in H^p(BH;Q)$ be nonzero classes. Let numbers $m_0,m_1,k_0,k_1 \in \bN_0$ be given with
\[
m_0 \geq 1, \; 4k_0 +p-2m_0 \geq 0, \; 2m_1+1=4k_1+q
\]
and let $m=m_0+m_1$ and $k=k_0 + k_1$. Then there is a bundle of closed oriented $(2m+1)$-manifolds $E \to X$ and a map $h: E \to B(G \times H)$ such that
\[
\kappa_{\cL_k,u \times v}(E,h)\neq 0 \in H^{4k+p+q-(2m+1)}(X;\bQ). 
\]
\end{mainthm}
Theorem \ref{thm:nontriviality} can be deduced without much effort from Galatius--Randal-Williams' work \cite{GRW14,GRW17} on diffeomorphism groups of even-dimensional manifolds; we give the short proof in \S \ref{sec:nontriviality}. Their work implies a nontriviality result for the $\kappa_{\cL_k,u}$-classes for bundles with even fibre dimension (see Lemma \ref{lem:grw-lemma} below); Theorem \ref{thm:nontriviality} follows by taking products with suitable odd-dimensional manifolds. 

To obtain Theorem \ref{mainthm:surfacegroup} (2), put $G=H=\bZ$, let $v=w \in H^1 (B \bZ;\bQ)$ be a generator, and let $m_0=m$, $k_0=k$, $m_1=m_1=0$. To obtain Theorem \ref{mainthm:surfacegroup} (1), put $G=1$, $H=\Gamma_g$, $u=1$ and $v=w$, as well as $m_0=m$, $k_0=k$, $k_1=m_1=0$.

To formulate the general statement behind Theorem \ref{mainthm:surfacegroup} (3), let $\Un(p,q)$ be the unitary group with signature $(p,q)$. The Lie group $\Un(p,q)$ contains $\Un(p) \times \Un(q)$ as a maximal compact subgroup, and so we have a map 
\begin{equation}\label{eqn:deltamap}
B\Un(p,q) \simeq B\Un(p) \times B\Un(q) \stackrel{\Delta}{\to} B\Un
\end{equation}
where the last map takes the difference of the two universal bundles. We denote by $\sch  \in \hat{H}^{2*} (B\Un(p,q);\bQ):= \prod_{k\geq 0} H^{2k} (B\Un(p,q);\bQ)$ (the \emph{super Chern character}) the pullback of the Chern character along the map \eqref{eqn:deltamap} and by $\sch_m$ its degree $2m$ component. Now let $B\Un(p,q)^{\delta}$ denote the classifying space of the discrete group $\Un(p,q)$; it comes with a natural map $B\Un(p,q)^\delta \to B\Un(p,q)$, and we denote the pullback of $\sch_m$ along this map still by the symbol $\sch_m \in H^{2m}(B\Un(p,q)^\delta;\bQ)$.
\begin{mainthm}\label{thm:vanishing}
Let $\pi:E \to X$ be a bundle of closed smooth oriented manifolds of odd dimension $n$, and let $f: E \to B  \Un(p,q)^\delta$ be a map. Then 
\[
\kappa_{\cL,\sch} (E,f)=0 \in H^{*-n}(X,\bQ).  
\]
\end{mainthm}

The deduction of Theorem \ref{mainthm:surfacegroup} (3) from Theorem \ref{thm:vanishing} is not entirely straightforward and is done in \S \ref{sec:surfacegroups} below; the key ingredient here is \cite{MilnorConnection}. 

The proof of Theorem \ref{thm:vanishing} is an elaboration of the index-theoretic proof that $\kappa_{\cL_k}=0$ given in \cite{Ebert13}. We replace the odd signature operator introduced in \cite{APSIII} and used in \cite{Ebert13} by a twisted version; this operator was introduced (on even-dimensional manifolds) by Lusztig \cite{Lusztig} in relation to the Novikov conjecture (surely it is already implicit in \cite{AtiyahSignature}); we give the construction in odd dimensions in \S \ref{subsect:twistedsignature}. The necessary cohomological index formula for the odd twisted signature operator is shown in \S \ref{subsec:indextheorem}. Here our procedure is somewhat different from the one in \cite{Ebert13}: the index theorem in the odd case is deduced from the even case, which is treated in adequate detail in \cite{AtiyahSingerIII}, \cite{AtiyahSingerIV} and \cite{LawsonMichelsohn}. This trick enforces to use Clifford algebras systematically. 

\section{Proof of the vanishing theorem}

\subsection{The family index of elliptic families}\label{sec:index-generalities}

As already said, the proof of Theorem \ref{thm:vanishing} is index-theoretic in nature, and in this section, we review the family index of a family of elliptic operators on a bundle of closed manifolds. Let $\pi:E \to X$ be a bundle of closed smooth manifolds and let $W \to E$ be a fibrewise smooth complex vector bundle with a fibrewise smooth bundle metric. We assume that the vertical tangent bundle $T_v E \to E$ is equipped with a fibrewise smooth Riemannian metric; then each fibre $E_x := \pi^{-1}(x)$ is a closed Riemannian manifold. Let $D$ be a family of formally selfadjoint elliptic operators of order $1$ acting on fibrewise smooth sections of $E$. We can restrict $D$ to each fibre and obtain a formally selfadjoint elliptic operator $D|_x$ acting on smooth sections of the bundle $W|_x:= W|_{E_x} \to E_x$. 

To define a family index of $D$, we shall firstly assume that $X$ is compact which is not a severe hypothesis for our purposes. Secondly, some linear algebraic symmetries on $D$ need to be present. 


\begin{defn}
A \emph{$\Cl^{k,0}$-structure} on a metric vector bundle $W\to E$ consists of a $\bZ/2$-grading $\iota$ on $W$ and a graded $*$-algebra homomorphism $\alpha: \Cl^{k,0} \to \End (TW)$. In some more detail, $\Cl^{k,0}$ denotes the Clifford algebra (with anticommuting generators $e_1,\ldots,e_k$ such that $e_i^2 =-1$) as a $\bZ/2$-graded algebra. The map $\alpha$ is determined by anticommuting skewadjoint operators $\alpha_j:= \alpha(e_j)$ which all anticommute with $\iota$ and have square $-1$.

A \emph{$\Cl^{k,0}$-graded elliptic family} over a manifold bundle $E\to X$ consists of a family $D$ of formally selfadjoint elliptic operators of order $1$, acting on the sections of some metric vector bundle $W \to E$, and a $\Cl^{k,0}$-structure $(\iota,\alpha)$ on $W$ such that the relations 
\[
D \iota + \iota D = 0 = \alpha(e_j) D + D \alpha(e_j)
\]
hold (for all $j=1,\ldots,k$). 
\end{defn}

A $\Cl^{k,0}$-graded elliptic family $(D,\iota,\alpha)$ over some bundle of closed smooth manifolds $\pi:E \to X$ over a finite CW base (this level of generality suffices for Theorem \ref{thm:vanishing}) has an index 
\[
\ind_k (D,\iota,\alpha) \in K^k (X). 
\]
The classical construction is essentially due to Atiyah--Singer and explained, with references, in \cite[\S 3]{Ebert13} in the case $k=0,1$ (to get it for arbitrary $k$, one invokes the general variant of the main result of \cite{AtiyahSingerSkew}). A more modern approach using Hilbert module techniques is described in \cite{JEIndex1}. 

A $\Cl^{0,0}$-graded elliptic family is essentially the same as a family of elliptic operators (no formal selfadjointness or further algebraic structure required); one splits $W = W_+ \oplus W_-$ into the eigenspaces of $\iota$ and looks at the restriction of $D$ to $W_+$ which maps to $W_-$. The index $\ind_0 (D,\iota)$ is just the ordinary family index of $D|_{W_+}$ from sections of $W_+$ to those of $W_-$. 

A $\Cl^{1,0}$-graded elliptic family is essentially the same as a family of formally selfadjoint elliptic operators without further algebraic symmetries; one looks at the restriction of $\alpha_1 D$ to the eigenbundle $W_+$ (this connects the present setup with the one used in \cite{Ebert13}). 

Let us describe the behavior of the index with respect to products. Given a $\Cl^{k,0}$-graded elliptic family $(D,\iota,\alpha)$ on some vector bundle $W\to E$ over $E \to X$ and a $\Cl^{l,0}$-graded elliptic family $(B,\eta,\beta)$ on some other vector bundle $W \to F$ over $F \to Y$, we form the exterior product of these elliptic families as follows. 

The underlying manifold bundle is $E \times F \to X \times Y$. Consider the bundle $V \boxtimes W := \proj_E^* V \otimes \proj_F^* W\to E \times F$ with grading $\iota \otimes \eta$. Here we have the graded tensor product differential operator $D \grotimes 1 + 1 \grotimes B$ (we use the graded tensor product here, with conventions as in \cite[\S 14.4]{Blackadar}), which is elliptic with square $D^2 \grotimes 1+1 \grotimes B^2$ and anticommutes with $\iota \otimes \eta$. We get a $\Cl^{k+l,0}$-structure $\alpha \sharp \beta$ by sending $e_i$ with $i \leq k$ to $\alpha(e_i) \grotimes 1$ and $e_i$ with $i>k$ to $1 \grotimes \beta (e_{i-k})$. Using the Koszul rule for the graded tensor product, it is easily checked that $D \grotimes 1 + 1 \grotimes B$ is $\Cl^{k+l,0}$-graded. The external product in $K$-theory is so that 
\begin{equation}\label{eqn:product-formula-for-index}
\ind_{k+l} (D \grotimes 1 + 1 \grotimes B,\iota \otimes \eta,\alpha \sharp \beta)= \ind_k (D,\iota,\alpha) \times \ind_l (B,\eta,\beta) \in K^{k+l}(X \times Y). 
\end{equation}

We need to know an explicit description of the Bott periodicity isomorphism $\bott: K^0 (X) \cong K^2 (X)$. The matrices
\[
\iota:=\twomatrix{1}{}{}{-1}, \, \beta(e_1):=\twomatrix{}{-1}{1}{}, \; \beta(e_2) := \twomatrix{}{i}{i}{} \in \mathrm{Mat}_{2,2}(\bC)
\]
define a $\Cl^{2,0}$-structure on the vector space $\bC^2$, viewed as a vector bundle over a point. 
The $0$ operator on $\bC^2$ is a $\Cl^{2,0}$-graded elliptic operator, and $\gb: =\ind_2 (0,\iota,\beta) \in K^2 (*) \cong \bZ$ is a generator. The Bott isomorphism is the map 
\[
\bott: K^0 (X) \cong K^2 (X), \; a \mapsto a \times \gb.
\]
Let $(D,\iota,\alpha)$ be some $\Cl^{2,0}$-graded elliptic family, acting on some vector bundle $W \to E$ over some fibre bundle $E \to X$. The class $\bott^{-1}(\ind_2 (D,\iota,\alpha)) \in K^0(X)$ has the following explicit description. Put $\varepsilon:= i \alpha (e_1 e_2) $, which is a selfadjoint involution on $W$ and commutes with $\iota$ and $D$. Hence we can restrict $D$ and $\iota$ to the eigenbundle $W_+=\Eig(\varepsilon,+1)$ and obtain a $\Cl^{0,0}$-graded elliptic family $(D|_{W_+},\iota|_{W_+})$. We have 
\begin{equation}\label{eqn:cliffordindex-bott}
\ind_2 (D,\iota,\alpha) = \bott (\ind_0 (D|_{W_+},\iota|_{W_+})). 
\end{equation}

We also make use of the Chern character, which is a map $\ch: K^k (X) \to \hat{H}^{k+2*}(X;\bQ)$. The Chern character is multiplicative in the sense that $\ch(a \times b) = \ch (a) \times \ch(b)$ whenever $a\in K^k (X)$ and $b \in K^l (Y)$. Moreover $\ch(\gb)=1 \in \hat{H}^{2*}(*;\bQ)=\bQ$. Both identities together imply 
\begin{equation}\label{eqn:Chern-Bott-commute}
\ch \circ \bott = \ch: K^0 (X) \to \hat{H}^{2*} (X;\bQ). 
\end{equation}

The following result, discovered in \cite[Proposition 4.1]{Hitchin} is the analytical basis for our vanishing theorem, see also \cite[Theorem 4.1]{Ebert13} for the proof. 

\begin{thm}\label{thm:abstract-vanishing-thm}
Assume that $(D,\iota,\alpha_1)$ is a $\Cl^{1,0}$-graded elliptic family and assume that the function $X \to \bN_0$, $x \mapsto \dim (\ker (D|_x))$, is locally constant. Then 
\[
\ind_1 (D,\iota,\alpha_1)=0 \in K^1 (X). 
\]
\end{thm}

One can give a more modern proof using Hilbert module theory; we refrain from giving it here.

\subsection{The twisted signature operator}\label{subsect:twistedsignature}

Let us first recapitulate the ordinary signature operator, assuming complex coefficients throughout. If $M$ is a closed Riemannian manifold of dimension $n$, we get an inner product on $\Lambda^* T^*M\otimes \bC$, and the exterior derivative $d:\cA^p (M) \to \cA^{p+1}(M)$ has a formal adjoint $d^*$. The operator $D=d+d^*: \cA^*(M) \to \cA^*(M)$ is elliptic and formally selfadjoint. By the general Hodge decomposition theorem, the kernel of $D: \cA^p (M) \to \cA^{p-1}(M)\oplus \cA^{p+1}(M)\subset \cA^*(M)$ can be identified with the de Rham cohomology $H^p_{dR}(M;\bC)\cong H^p(M;\bC)$. 
Let $\iota$ be the even/odd grading on $\Lambda^* T^* M$; as $D$ maps even to odd forms and vice versa
\begin{equation}\label{eqn:Dgraded}
D \iota + \iota D =0, 
\end{equation}
so that $(D,\iota)$ is $\Cl^{0,0}$-graded; and $\ind(D,\iota)=\chi(M)\in \bZ=K^0(*)$ is the Euler number of $M$.

If $M$ is oriented, we have the Hodge star operator $\star: \Lambda^p T^* M \to \Lambda^{n-p} T^* M$; with the help of $\star$, we can express the $L^2$-inner product on forms by 
\[
\scpr{\alpha,\beta}_{L^2} = \int_M \alpha \wedge \star \beta .
\]
Note that $\star(1)$ is the volume form, and recall that on $p$-forms, 
\begin{equation}\label{eqn:starsquare}
\star^2 = (-1)^{p(n-p)}.
\end{equation}
It follows from Stokes' theorem and \eqref{eqn:starsquare} that the adjoint $d^*: \cA^p (M) \to \cA^{p-1}(M)$ is given by the formula
\begin{equation}\label{eqn:dstar}
d^* =  (-1)^{np-n+1} \star  d \star.
\end{equation}
We modify the Hodge star and define on $p$-forms
\begin{equation}\label{eqn:defsigma}
\tau := i^{\frac{n(n+1)}{2} + 2np+ p(p-1)} \star. 
\end{equation}
Equations \eqref{eqn:starsquare}, \eqref{eqn:dstar} and \eqref{eqn:defsigma} imply
\begin{lem}\label{lem:signatureoperator}
The following relations hold.
\begin{enumerate}
\item $\tau^2=1$,
\item $\iota \tau=(-1)^n \tau \iota$, 
\item $d^* =(-1)^{n+1} \tau d \tau$, $d \tau = (-1)^{n+1} \tau d^*$, $ \tau d = (-1)^{n+1} d^* \tau $,
\item $D\tau = (-1)^{n+1} \tau D$.
\end{enumerate}
\end{lem}

\begin{defn}
For each $n$, we call the $\Cl^{0,0}$-graded elliptic operator $(D,\iota)$ the \emph{Euler characteristic operator}. For even $n$, we call the $\Cl^{0,0}$-graded elliptic operator $(D,\tau)$ the \emph{signature operator}. For odd $n$, we use the grading $\iota$ and define a Clifford structure $\alpha : \Cl^{1,0} \to \End (\Lambda^* T^*M)$ by $\alpha (e_1):= \iota \tau$ and call the resulting $\Cl^{1,0}$-graded elliptic operator $(D,\iota,\alpha)$ the \emph{odd signature operator}. 
\end{defn}

We continue to assume that $M$ is a closed oriented $n$-dimensional Riemannian manifold. In addition, let us suppose that $V \to M$ is a complex vector bundle, together with a nondegenerate (possibly indefinite) hermitian form $\eta: V \times V \to \bC$ and a connection $\nabla$ with respect to which $\eta$ is parallel.

\begin{lem}\label{lem:compatiblepair}
There exist pairs $(h,\sigma)$, consisting of a smooth bundle metric $h$ on $V$ and an $h$-isometry $\sigma$ with $\sigma^2=1$ and
\[
h(\_,\_)=\eta (\_,\sigma\_). 
\]
We can choose $(h,\sigma)$ to depend continuously on a smooth bundle metric $h_0$. Such a pair is called a \emph{compatible pair}.
\end{lem}

\begin{proof}
Pick a smooth bundle metric $h_0$ at random, which determines uniquely a bundle endomorphism $S$ such that $h_0(S\_,\_)=\eta(\_,\_)$, which is selfadjoint and invertible. Let 
\[
h(\_,\_):= h_0 (|S|\_,\_), \; \sigma:= S|S|^{-1}= |S|^{-1} S;
\]
one checks that $(h,\sigma)$ is a compatible pair.
\end{proof}

We can extend $\nabla$ to $d_\nabla: \cA^*(M;V) \to \cA^{*+1}(M;V)$ on $V$-valued differential forms.  The condition that $d_\nabla^2=0$ is equivalent to $\nabla$ being flat. If that is the case,  
\begin{equation}\label{eqn:ellipticcomplex}
0 \to \cA^0 (M;V) \stackrel{d_\nabla}{\to} \ldots \stackrel{d_\nabla}{\to} \cA^n (M;V) \to 0
\end{equation}
is an elliptic complex. A twisted version of de Rham's theorem \cite[\S 1]{Mueller} \cite[\S 1]{Hattori} states that 
\begin{equation}\label{eqn:twistedderahm}
H^* (\cA^*(M;V),d_\nabla) \cong H^* (M;V); 
\end{equation}
here the right hand side is singular cohomology with coefficient system defined by the flat bundle $V$.

Now fix a compatible pair $(h,\sigma)$; note that in general $h$ and $\sigma$ will not be parallel for $\nabla$. The bundle metric $h$, together with the original Riemannian metric on $M$, determines a bundle metric on $\Lambda^* T^* M \otimes V$. With respect to this metric, we form the formal adjoint $d_\nabla^*$ of $d_\nabla$, and let 
\begin{equation}\label{eqn:defnDnabla}
D_{V}: = d_\nabla + d_\nabla^*
\end{equation}
(regardless of $\nabla$ being flat). The principal symbol of $D_V$ is given by the formula $\smb(D_V) = \smb (D) \otimes 1_V$, so $D_V$ is elliptic. A $\bZ/2$-grading on $\Lambda^* T^*M \otimes V$ is given by $\iota_V := \iota \otimes 1_V$, and we have 
\[
D_V \iota_V + \iota_V D_V =0,
\]
generalizing \eqref{eqn:Dgraded}.
If $\nabla$ is flat, we can give a cohomological interpretation of $\ker (D_V)$: the Hodge theorem for elliptic complexes \cite[Theorem 5.5.2]{Wells} \cite[Theorem 6.2]{Roe} applied to \eqref{eqn:ellipticcomplex} proves that 
\begin{equation}\label{eqn:kernel-twistedsignature}
\ker (D_V) \cong H^* (\cA^*(M;V),d_\nabla).
\end{equation}

The hermitian form $\eta$ can be viewed as a bundle map $\eta: V \otimes_\bR V \to \bC$. So $1_{\Lambda^* TM} \otimes \eta$ induces a map 
\[
\eta: \cA^* (M;V\otimes_\bR V) \to \cA^* (M)
\]
and the same construction can be applied to $h$ instead. 
Assuming that $M$ is oriented, the inner product on $V$-valued forms is given by 
\begin{equation}\label{eqn:scalarproductvshdoge0}
\scpr{\alpha,\beta} = \int_M h (\alpha \wedge (\star\otimes 1_V) \beta);
\end{equation}
to read this formula, note that $\alpha \wedge (\star\otimes 1_V) \beta  \in \cA^n (M;V \otimes_\bR V)$. When $\alpha = a \otimes v$ and $\beta = b \otimes w$ where $a,b \in \cA^*(M)$ and $v,w\in C^\infty(M;V)$, we have 
\[
h((a \otimes v) \wedge ((\star\otimes 1_V) b \otimes w)) = 
h((a \wedge \star b) \otimes (v \otimes w)) = 
(a \wedge \star b) \otimes h(v,w)=
\]
(using that $(h,\sigma)$ are compatible)
\[
(a \wedge \star b) \otimes \eta(v,\sigma w) = \eta ((a \otimes v) \wedge ((\star \otimes \sigma) (b \otimes w)). 
\]
Since any vector-valued form can be written as linear combination of such elementary tensors, we obtain from \eqref{eqn:scalarproductvshdoge0} the formula
\begin{equation}\label{eqn:scalarproductvshdoge}
\scpr{\alpha,\beta} = \int_M \eta (\alpha \wedge (\star\otimes \sigma) \beta).
\end{equation}

\begin{lem}\label{lem:adjointofdnablatwisted}
The formal adjoint 
\[
d_\nabla^*: \cA^{p} (M;V) \to \cA^{p-1}(M;V)
\]
is given by the formula
\[
d_\nabla^* =  (-1)^{np-n+1} (\star \otimes \sigma) d_\nabla (\star\otimes \sigma).
\]
\end{lem}

\begin{proof}
Let $\alpha \in \cA^{p-1}(M;V)$ and $\beta \in \cA^p (M;V)$. By Stokes' theorem
\[
0 = \int_M d \eta (\alpha \wedge (*\otimes \sigma)\beta). 
\]
Since $\eta$ is $\nabla$-parallel
\[
d \eta (\alpha \wedge (*\otimes \sigma)\beta) = \eta ((d_\nabla \alpha) \wedge (* \otimes \sigma) \beta) + (-1)^{p-1}\eta (\alpha \wedge d_\nabla (* \otimes \sigma) \beta = 
\]
\[
\eta ((d_\nabla \alpha) \wedge (* \otimes \sigma) \beta) - (-1)^p \eta (\alpha \wedge( (* \otimes \sigma) (* \otimes \sigma)^{-1}  d_\nabla (* \otimes \sigma) \beta.
\]
So \eqref{eqn:scalarproductvshdoge} implies that
\[
\scpr{d_\nabla \alpha,\beta} = (-1)^p \scpr{\alpha, (* \otimes \sigma)^{-1}  d_\nabla (* \otimes \sigma) \beta}
\]
or 
\[
d_\nabla^* \beta =(-1)^p (* \otimes \sigma)^{-1}  d_\nabla (* \otimes \sigma) \beta \stackrel{\eqref{eqn:starsquare}}{=} 
(-1)^{np-n+1} (\star \otimes \sigma) d_\nabla (\star\otimes \sigma) \beta.
\]
\end{proof}

The proof of Lemma \ref{lem:signatureoperator} was a formal consequence of \eqref{eqn:starsquare}, \eqref{eqn:dstar} and \eqref{eqn:defsigma}. Hence we obtain the following generalization of Lemma \ref{lem:signatureoperator}, with the same proof. We introduce the involution $\tau_V:= \tau \otimes \sigma$ of $\Lambda^* T^* M \otimes V$.

\begin{prop}\label{prop:twistedsignatureoperator}
Let $M^n$ be an oriented Riemannian manifold and let $V \to M$ be a complex vector bundle with a flat connection $\nabla$ and a parallel nondegenerate hermitian form $\eta$. Pick a compatible pair $(h,\sigma)$. Let $D_V := d_\nabla + d_\nabla^*$. Then 
\begin{enumerate}
\item $\iota_V^2=1$, $\tau_V^2=1$ and 
\[
\iota_V \tau_V=(-1)^n  \tau_V \iota_V. 
\]
\item 
\[
D_V \iota_V + \iota_V  D_V =0
\]
\item 
\[
D_V \tau_V + (-1)^{n} \tau_V D_V =0.
\]
\end{enumerate}\qed
\end{prop}

\begin{defn}
In the situation of Proposition \ref{prop:twistedsignatureoperator}, we make the following definitions. 
\begin{enumerate}
\item The $\bZ/2$-graded elliptic operator $(D_V,\iota_V)$ is the \emph{twisted Euler characteristic operator}.
\item If $n$ is even, the $\bZ/2$-graded elliptic operator $(D_V,\tau_V)$ is the \emph{even twisted signature operator}.
\item If $n$ is odd, we define $\alpha: \Cl^{1,0} \to \End (\Lambda^* T^* M \otimes V)$ by $\alpha (e_1)= \iota_V \tau_V$. The $\Cl^{1,0}$-elliptic operator $(D_V,\iota_V,\alpha)$ is called the \emph{odd twisted signature operator}.
\end{enumerate}
\end{defn}

The index of the twisted Euler characteristic operator is 
\[
\ind(D_V,\iota_V)= \sum_{p=0}^{n} (-1)^p\dim H^p (M;V) = \chi(M) \rank (V) \in \bZ,
\]
and when $n=2m$, the index of the even twisted signature operator is the signature of the following hermitian form on $H^m (M;V)$: 
\[
(\alpha,\beta) \mapsto i^{m^2}\int_M \eta (\alpha \wedge \beta).
\]
This is shown by the same argument that identifies the index of the signature operator with the signature of $M$. Compare \cite[p. 575 f]{AtiyahSingerIII}, or follow the computation: 
\[
\ind (D_V,\tau_V)  = \Tr (\tau_V|_{\ker (D_V)}) = \Tr (\tau_V|_{\ker (D_V)\cap \cA^m(M;V)}).
\]
The second equation holds since $\tau_V$ decomposes into maps $\cA^p(M;V) \to \cA^{2m-p}(M;V)$ and hence only $\tau_V|_{\cA^m (M;V)}$ contributes to the trace. The latter is the signature of the hermitian form $\scpr{\_,\tau_V \_}$ on $\ker (D_V)\cap \cA^m(M;V)$. However
\[
\scpr{\alpha,\tau_V \beta} \stackrel{\eqref{eqn:scalarproductvshdoge}}{=} \int_M \eta (\alpha \wedge (\star\otimes \sigma) \tau_V \beta) =
\]
\[
= i^{\frac{2m(2m+1)}{2} + 4m^2+ m(m-1)} (-1)^{m^2} \int_M \eta (\alpha \wedge \beta) .
\]

\begin{rem}
When $V$ is a complex vector bundle with a nondegenerate hermitian form $\eta$ and $\nabla$ is a not necessarily flat connection leaving $\eta$ parallel, we can still pick a compatible pair $(h,\sigma)$ and define involutions $\iota_V$ and $\tau_V$, satisfying the identity (1) of Proposition \ref{prop:twistedsignatureoperator}. The operator $D_V:= d_\nabla + d_\nabla^*$ still satisfies (2). However, (3) only holds approximately in the sense that $D_V \tau_V + (-1)^{n} \tau_V D_V$ has order zero. Then $\tilde{D}_V := \frac{1}{2} (D_V + (-1)^{n+1}\tau_V D_V \tau_V$ satisfies both (2) and (3) and has the same principal symbol as $D_V$. The discussion of the topological index carried out in the next subsection carries over to this more general situation. However, the analytical index of $\tilde{D}_V$ does not have a direct cohomological interpretation.
\end{rem}

The whole discussion carries over to families. Let $\pi:E \to X$ be a bundle of smooth closed oriented Riemannian $n$-manifolds. We say that a \emph{fibrewise flat hermitian bundle} $(V,\eta,\nabla) \to E$ is a complex vector bundle with a hermitian form $\eta$ and a family of connections $\nabla$ on $V$ such that each $\nabla|_x:= \nabla|_{E_x}$ is flat and such that $h|_{E_x}$ is $\nabla|_x$-parallel for each $x$. This means each $V|_{E_x}$ comes with a reduction of structure group to $\Un(p,q)^\delta$. We say that $(V,\eta,\nabla)$ is \emph{globally flat} if $V$ comes with a reduction of structure group to $\Un(p,q)^\delta$ inducing the family of connections $\nabla$. 

Note that being globally flat is a much stronger condition than being fibrewise flat. Before Lemma \ref{lem:spectralflow} below, we describe an example (due to \cite[\S 4]{Lusztig}) of a fibrewise flat hermitian bundle that cannot be globally flat. 

If $(V,\eta,\nabla)$ is fibrewise flat, we can pick a compatible pair $(h,\sigma)$. In that case, we can define the twisted operator family $D_V$; the restriction to each fibre is $D_V|_{E_x}$. We also have the two involutions $\iota_V$ and $\tau_V$; and all the identities of Proposition \ref{prop:twistedsignatureoperator} continue to hold. 
We define family indices 
\[
\chi(E\to X;V) := \ind_0 (D_V,\iota_V) \in K^0 (X),
\]
and for $n$ even
\[
\sign_\even (E \to X;V) := \ind_0 (D_V,\tau_V)\in K^0 (X).
\]
For $n$ odd, we define 
\[
\sign_\odd (E \to X;V) := \ind_1 (D_V, \iota_V, \alpha(e_1):= \iota_V \tau_V ) \in K^1 (X).
\]
\begin{cor}\label{cor:indextwistedsignature-trivial}
Let $E \to X$ be a bundle of smooth closed oriented odd-dimensional manifolds and let $(V,\eta) \to E$ be globally flat hermitian vector bundle. Then 
\[
\sign_\odd (E \to X;V) = 0  \in K^1 (X).
\]
\end{cor}

\begin{proof}
The dimension of the kernel of $(D_V)|_x$ is equal to $\sum_{k=0}^n \dim H^k (E_x,V_x)$. Since $V$ is globally flat, this is a locally constant function of $x$. Hence Theorem \ref{thm:abstract-vanishing-thm} applies and gives the desired conclusion.
\end{proof}

\subsection{The index formula for the twisted signature operator}\label{subsec:indextheorem}

In this subsection, we derive the cohomological formula for the index of the odd twisted signature operator, which reads as follows. 

\begin{thm}[Index theorem for the odd signature operator]\label{thm:indextheorem-odd}
Let $\pi:E \to X$ be a bundle of closed oriented $2m+1$-dimensional manifolds over a finite CW complex. 
Let $(V,\eta) \to E$ be a fibrewise flat hermitian vector bundle. Then 
\[
\ch (\sign_\odd (E \stackrel{\pi}{\to} X;V)) = \pm 2^m \pi_! (\cL (T_v E) \sch (V)) \in H^{2*+1}(X;\bQ).
\]
\end{thm}

\begin{proof}[Proof of Theorem \ref{thm:vanishing}]
Let $f: E \to B  \Un(p,q)^\delta$ be a map and let $V \to E$ be the (globally) flat hermitian vector bundle classified by $f$. Then $\kappa_{\cL,\sch}(E,f)=\pi_! (\cL (T_v E) \sch (V))$. As $V$ is globally flat, Corollary \ref{cor:indextwistedsignature-trivial} shows $\sign_\odd (E \stackrel{\pi}{\to} X;V)=0$; so Theorem \ref{thm:indextheorem-odd} implies Theorem \ref{thm:vanishing}.
\end{proof}

\begin{rem}
The sign (which depends only on $m$) in Theorem \ref{thm:indextheorem-odd} does not matter for our overall purpose, and the author is too lazy to figure out a sign that he can afford to leave undetermined.
The proof of Lemma \ref{lem:spectralflow} below uses a squaring trick; any attempt to figure out the sign will depend on a different proof of that Lemma, which can be done using a spectral flow argument (still depending on keeping track of many sign conventions).
\end{rem}

The first step in the proof is the cohomological index formula for the twisted signature operator in even dimensions. 

\begin{thm}\label{thm:index-even-sgnature}
If $n=2m$, we have 
\[
\ch (\sign_\even (E \stackrel{\pi}{\to} X;V)) = (-1)^m 2^m \pi_! (\cL (T_v E) \sch (V)) \in H^{2*}(X;\bQ).
\]
\end{thm}
This is straightforward from the usual family index theorem, see \cite[p.577]{AtiyahSingerIII} for the derivation of the cohomological formula (the sign $(-1)^m$ is not present in loc.cit; it appears here because the involution $\tau$ \eqref{eqn:defsigma} which appeared more natural to us differs from the involution with the same name in \cite[p. 575]{AtiyahSingerIII} by the sign $(-1)^m$).
As next step, we need two facts about the twisted Euler characteristic operator.

\begin{prop}\label{prop:eulercharacteristic}
\mbox{}
\begin{enumerate}
\item For two bundles of closed manifolds $\pi_i: E_i \to X_i$ and fibrewise flat hermitian vector bundles $(V_i,\eta_i) \to E_i$, we have 
\[
\chi(E_0 \times E_1 \to X_0 \times X_1;V_0 \boxtimes V_1) = \chi (E_0 \to X_0,V_0) \times \chi(E_1 \to X_1, V_1) \in K^0 (X_0 \times X_1). 
\]
\item For a bundle $\pi:E \to X$ of closed, oriented and odd-dimensional manifolds and a fibrewise flat hermitian vector bundle $(V,\eta) \to E$, we have
\[
\chi(E \to X;V)=0 \in K^0(X). 
\]
\end{enumerate}
\end{prop}

\begin{proof}
(1): Let us adopt the convention that we use the $\boxtimes$-symbol and the graded variant $\grboxtimes$ when we tensor operators on different manifolds (bundles), and reserve $\otimes$ and $\grotimes$ for the tensor product of bundle endomorphisms on an individual manifold (bundle). We only spell out the computation when $E_i=M_i \to *$, the general case is only notationally more involved. 

The operator $D_{V_0} \grboxtimes 1 + 1 \grboxtimes D_{V_1}$ acts on the vector bundle $\Lambda^* T^*(M_0 \times M_1) \otimes V_0 \boxtimes V_1)$; here we use the usual graded isomorphism 
\begin{equation}\label{eqn:isomorphisms-exterioralgebra}
\Lambda^* T^* M_0 \hat{\boxtimes} \Lambda^* T^* M_1 \cong \Lambda^* T^* (M_0 \times M_1). 
\end{equation}
As the grading of each $D_{V_i}$ is given the even-odd grading on forms, $D_{V_0} \grboxtimes 1 + 1 \grboxtimes D_{V_1}$ agrees as an operator with $D_{V_0 \boxtimes V_1}$ (using the tensor product connection, metric and hermitian form on $V_0 \boxtimes V_1$). The grading is the even/odd grading. Hence the claim is a consequence of the product formula \eqref{eqn:product-formula-for-index}.

(2): the operator $i \iota_V \tau_V$ is odd, anticommutes with $D_V$ and has square $+1$. By the homotopy invariance of the index, we get that 
\[
\chi(E \to X;V) = \ind_0 (D_{V}+ i t\iota_V \tau_V,\iota_V)
\]
for each $t \in \bR$, but $(D_{V}+ i \iota_V \tau_V)^2 = D_V^2 +1$ is invertible. Hence $\chi(E \to X;V) =0 \in K^0 (X)$.
\end{proof}

\begin{prop}\label{prop:product-formula-odd-signatureoperator}
Let $\pi_i: E_i \to X_i$, $i=0,1$, be two bundles of closed oriented manifolds, and let $(V_i,\eta_i) \to E_i$ be two fibrewise flat hermitian vector bundles. If the fibre dimensions $n_i=2m_i+1$ are both odd, we have the identity
\[
\sign_\even(E_0 \times E_1 \to X_0 \times X_1;V_0 \boxtimes V_1) = 
\]
\[
=2 (-1)^{m_0+m_1} \bott^{-1}(\sign_\odd (E_0 \to X_0;V_0) \times \sign_\odd (E_1\to X_1;V_1)) 
\]
in the group $K^0 (X_0 \times X_1)[\frac{1}{2}]$. 
\end{prop}

In \cite[Lemma 6]{RosWein}, analogous formulas for the $K$-homology classes of the signature operators are established. These do not quite imply Proposition \ref{prop:product-formula-odd-signatureoperator} (but a family version of \cite[Lemma 6]{RosWein} which should not be too hard would do it). In loc.cit., these formulas are stated without inverting $2$. However, the author has been unable to understand the details of the proof given in \cite{RosWein}, so he prefers to give an independent treatment. 

The linear algebra underlying the signature operators can be simplified by the use of Clifford algebras (though the connection to cohomology requires to use the exterior algebra). For the proof of Proposition \ref{prop:product-formula-odd-signatureoperator}, we have to recall how this works. 

Let $T$ be a euclidean vector space of dimension $n$, and consider the usual Clifford action $c: \Cl (T) \to \End(\Lambda^* T)$, which sends $e \in T$ to $\ext_e - \ext_e^*$ ($\ext_e$ is exterior multiplication with $e$, and its adjoint is the insertion operator). Let $\iota$ denote the even/odd grading on $\Lambda^* T$. 

Assuming that $T$ is oriented, pick an oriented orthonormal basis $(e_1,\ldots,e_n)$, and let $\vol_T \in \Lambda^n T$ be the volume form. 
The \emph{Clifford volume element} is defined as $\omega_T:= e_1 \cdots e_n \in \Cl(T)$.

\begin{lem}\label{lem:cliffordvolume}
\mbox{}
\begin{enumerate}
\item $\omega_T^2 = (-1)^{\frac{n(n+1)}{2}}$.
\item The relation to the Hodge star operator is given by the identity
\[
c(\omega_T) = (-1)^{\frac{p(p-1)}{2} +np} \star: \Lambda^p (T) \to \Lambda^{n-p} (T) . 
\]
\item $\tau = i^{\frac{n(n+1)}{2}} c(\omega_T)$.
\end{enumerate}
\end{lem}

\begin{proof}
(1) is straightforward, (2) is shown in \cite[p.129]{LawsonMichelsohn} and (3) is trivial.
\end{proof}

\begin{proof}[Proof of Proposition \ref{prop:product-formula-odd-signatureoperator}]
We first express 
\[
\mathbf{s}:= \sign_\odd (E_0 \to X_0;V_0) \times \sign_\odd (E_1\to X_1;V_1) \in K^2 (X_0 \times X_1)
\]
as the $\Cl^{2,0}$-index of another operator family on $E_0 \times E_1 \to X_0 \times X_1$. 
Computing exactly as in the proof of Proposition \ref{prop:eulercharacteristic} (1) (and using that the grading of the odd twisted signature operator is the even/odd grading, as it is the case for the twisted Euler characteristic operator), we find that $\mathbf{s}$ is the $\Cl^{2,0}$-index of the operator $D_{V_0} \grboxtimes 1 + 1 \grboxtimes D_{V_1}$ acting on the vector bundle $\Lambda^* T_v^*(E_0 \times E_1) \otimes V_0 \boxtimes V_1$ with the grading $\iota_{V_0 \boxtimes V_1}$. The $\Cl^{2,0}$-structure is given by the two generators $\alpha_0 \grboxtimes 1$, $1 \grboxtimes \alpha_1$, where 
\[
\alpha_i=\iota_{V_i}\tau_{V_i}.
\] 
Next we calculate 
\[
\bott^{-1}(\mathbf{s})\in K^0 (X_0 \times X_1) 
\]
by the recipe \eqref{eqn:cliffordindex-bott}. Using the notation introduced before that formula, the involution $\varepsilon$ is given by
\[
\varepsilon = i (\alpha_0 \grboxtimes 1)(1 \grboxtimes \alpha_1) = i \iota_{V_0} \tau_{V_0} \grboxtimes \iota_{V_1} \tau_{V_1}. 
\]
Keeping in mind the Koszul rule, the definition of $\tau_{V_i}$ and that $\tau_{V_0}$ is even, we rewrite this as
\begin{equation}\label{productformula:eq1}
\varepsilon= i (\iota_0 \grboxtimes \iota_1) (\tau_0 \grboxtimes \tau_1) (\sigma_0 \grboxtimes \sigma_1),
\end{equation}
with the following interpretation understood: $\iota_j$ is the even/odd-grading on $\Lambda^* T_v E_j$, tensored with the identity on $V_j$, $\tau_j$ is the $\tau$-operator $\Lambda^* T_v E_j$, tensored with the identity on $V_j$, and $\sigma_j$ is the compatible involution on $V_j$, tensored with the identity on $\Lambda^* T_v E_j$. 

Under the isomorphism \eqref{eqn:isomorphisms-exterioralgebra}, $\iota_0 \grboxtimes \iota_1$ corresponds to the even/odd grading $\iota$ on $\Lambda^* T_v^* (E_0 \times E_1)$, tensored with the identity on $V:= V_0 \boxtimes V_1$, and $\sigma_0 \grboxtimes \sigma_1$ is the compatible involution $\sigma$ on $V$, tensored with the identity on the exterior algebra. 

Next recall Lemma \ref{lem:cliffordvolume}. If $\omega_i$ is the volume element in the Clifford algebra of $T_v E_i$, we see that 
\begin{equation}\label{productformula:eq2}
(\tau_0 \grboxtimes \tau_1) = i^{\frac{(2m_0+1)(2m_0+2)}{2} + \frac{(2m_1+1)(2m_1+2)}{2}} c(\omega_0) \grboxtimes c(\omega_1). 
\end{equation}
The term $c(\omega_0) \grboxtimes c(\omega_1)$ corresponds under \eqref{eqn:isomorphisms-exterioralgebra} to $c(\omega)$, with $\omega$ the volume element in the Clifford algebra of $T_v (E_0 \times E_1)$ (using the standard orientation conventions for products). Using Lemma \ref{lem:cliffordvolume} (3), we get 
\begin{equation}\label{productformula:eq5}
(\tau_0 \grboxtimes \tau_1) =  
i^{\frac{(2m_0+1)(2m_0+2)}{2} + \frac{(2m_1+1)(2m_1+2)}{2}} c(\omega) =
 i (-1)^{m_0+m_1+1} \tau. 
\end{equation}
Putting \eqref{productformula:eq1}, \eqref{productformula:eq2} and \eqref{productformula:eq5} together, we obtain 
\begin{equation}\label{productformula:eq6}
\varepsilon =  (-1)^{m_0+m_1} \iota \tau \sigma = (-1)^{m_0+m_1} \iota_{V} \tau_{V}. 
\end{equation}
Formula \eqref{eqn:cliffordindex-bott} hence shows that 
\begin{equation}\label{productformula:eq9}
\bott^{-1}(\mathbf{s})=
\ind_0 (D_{V_0\boxtimes V_1}|_{\Eig((-1)^{m_0+m_1} \iota_{V} \tau_{V},+1)},\iota_V) \in K^0 (X_0 \times X_1). 
\end{equation}

To relate this to the index of the twisted signature operator on the product, we contemplate on the following construction. Let $\pi:E \to X$ be a bundle of closed, oriented $n$-manifolds and let $V \to E$ be a fibrewise flat hermitian vector bundle. When $n$ is \emph{even}, the twisted Euler characteristic and twisted signature operator admit a further decomposition as follows. 

The operator $\iota_V \tau_V$ is an involution which commutes with $\iota_V$, $\tau_V$ and $D_V$. We thus can decompose $\Lambda^* T_v^* E \otimes V = W_+ \oplus W_-$ into eigenbundles of $\iota_V \tau_V$; the three operators $\iota_V$, $\tau_V$ and $D_V$ preserve this decomposition. 
We can therefore write the $K^0$-index of the twisted Euler characteristic operator as 
\[
\chi(E \to X;V)= \ind_0 (D_V|_{W_+},\iota_V) + \ind_0 (D_V|_{W_-},\iota_V)
\]
and that of the twisted signature operator as 
\[
\sign_\even(E \to X;V)= \ind_0 (D_V|_{W_+},\tau_V) + \ind_0 (D_V|_{W_-},\tau_V). 
\]
On $W_\pm$, we have $\tau_V = \pm \iota_V$. This, together with the general relation $\ind_0 (D,-\iota)=-\ind_0(D,\iota)$ for $\Cl^{0,0}$-graded elliptic families, shows that 
\[
\sign_\even(E \to X;V)= \ind_0 (D_V|_{W_+},\iota_V) - \ind_0 (D_V|_{W_-},\iota_V). 
\]
Using the notation 
\[
\midex_\pm (E \to X ;V):= \ind_0 (D_V|_{W_\pm},\iota_V),
\]
we hence have 
\begin{equation}\label{eqn:decompositionEulerSignature}
\chi(E \to X ;V)= \midex_+ (E \to X ;V)  + \midex_- (E \to X ;V) \in K^0(X)
\end{equation}
and
\begin{equation}\label{eqn:decompositionEulerSignature2}
\sign_\even(E \to X;V) = \midex_+ (E \to X;V) -\midex_- (E \to X ;V) \in K^0(X).
\end{equation}
Inverting $2$, we can solve \eqref{eqn:decompositionEulerSignature} and \eqref{eqn:decompositionEulerSignature2} for $\midex_\pm$ and obtain 
\begin{equation}\label{eqn:decompositionEulerSignature1}
 \midex_\pm (E \to X ;V) = \frac{1}{2} (\chi(E \to X ;V)\pm \sign_\even(E \to X;V))\in K^0(X)[\frac{1}{2}].
\end{equation}

The formula \eqref{productformula:eq9} can be restated by writing
\begin{equation}\label{productformula:eq8}
\bott^{-1}(\mathbf{s}) = \midex_{(-1)^{m_0+m_1}} (E_0 \times E_1 \to X_0 \times X_1 ;V_0 \boxtimes V_1). 
\end{equation}

Inserting \eqref{eqn:decompositionEulerSignature1} into \eqref{productformula:eq8} therefore leads to the formula $2 \bott^{-1}(\mathbf{s}) = $
\[
\chi(E_0 \times E_1 \to X_0 \times X_1;V)+(-1)^{m_0+m_1}\sign_\even(E_0 \times E_1 \to X_0 \times X_1;V) 
\]
in $K^0 (X_0 \times X_1)[\frac{1}{2}]$. By Proposition \ref{prop:eulercharacteristic}, the first summand vanishes, and the proof is complete.
\end{proof}


To proceed with the proof of Theorem \ref{thm:indextheorem-odd}, we consider an explicit example, which is due to Lusztig \cite[\S 4]{Lusztig}. Consider the trivial fibre bundle $\proj_1: S^1 \times S^1 \to S^1 $. In loc.cit, it is shown that the complex line bundle $L \to S^1 \times S^1$ whose first Chern class is the generator $u  \times u \in H^2(S^1 \times S^1;\bZ)$ ($u \in H^1 (S^1)$ is an arbitrary generator) has a positive definite hermitian metric and a hermitian connection $\nabla$ whose restriction to each fibre $\proj_1^{-1}(z)$ is flat (one should think of $L|_{\{z \}\times S^1}$ as the flat line bundle with monodromy $z$). 

\begin{lem}\label{lem:spectralflow}
The analytical index $\sign_\odd (S^1 \times S^1 \stackrel{\proj_1}{\to} S^1;L) \in K^1 (S^1) \cong \bZ$ is a generator, and $\scpr{(\proj_1)_! (\cL (T_v S^1\times S^1) \ch(L)),[S^1]} = \pm 1 \in \bQ$.
\end{lem}

\begin{proof}
We have $c_1 (L)=u \times u$. The vertical tangent bundle of the projection $\proj_1$ is trivial; hence
\[
(\proj_1)_! (\cL (T_v S^1\times S^1) \ch(L)) = (\proj_1)_! ( 1+ u \times u) = \pm u\in H^1 (S^1;\bZ)
\]
is a generator. We compute the analytical index (up to a sign) with the help of Proposition \ref{prop:product-formula-odd-signatureoperator}. By that formula (and since $\bott$ is an isomorphism), it is sufficient to check that 
\begin{equation}\label{eqn:luzstig-family-square}
\sign_\even (S^1 \times S^1 \times S^1 \times S^1 \stackrel{\proj_1 \times \proj_3}{\to} S^1 \times S^1 ; L \boxtimes L) \in K^0 (S^1 \times S^1)
\end{equation}
lies in the reduced group $\tilde{K^0} (S^1 \times S^1) \cong \bZ$ and is twice a generator. The Chern character on $S^1 \times S^1$ takes values in integral cohomology and is an isomorphism $K^0 (S^1\times S^1 ) \to H^{2*}(S^1 \times S^1;\bZ)$. So we must show that the Chern character of \eqref{eqn:luzstig-family-square} is twice a generator of $H^2(S^1 \times S^1;\bZ)$. This can be done with the help of Theorem \ref{thm:index-even-sgnature}: 
\[
\ch (\sign_\even (S^1 \times S^1 \times S^1 \times S^1 \stackrel{\proj_1 \times \proj_3}{\to} S^1 \times S^1 ; L \boxtimes L)) = 
\]
\[
-2 (\proj_1 \times \proj_3)_! ( \sch (L \boxtimes L)),
\]
using that the vertical tangent bundle is trivial. Now $L$ and $L \boxtimes L$ have trivial grading, so 
\[
\sch (L \boxtimes L) = \ch  (L \boxtimes L) = \ch(L) \times \ch(L) = (1 +u \times u)\times (1 +u \times u) = 
\]
\[
= 1 + (u \times u \times 1 \times 1) + (1 \times 1 \times u \times u) + (u \times u \times u \times u).
\]
Applying $-2 (\proj_1 \times \proj_3)_!$ yields as degree $0$ component $0$, and as degree $2$ component
\[
-2 (\proj_1 \times \proj_3)_!(u \times u \times u \times u)= \pm 2 u \times u,
\]
as claimed. 
\end{proof}

\begin{proof}[Proof of Theorem \ref{thm:indextheorem-odd}]
Let $\pi: E \to X$ be a bundle with $2m-1$-dimensional fibres and let $V \to E$ be a fibrewise flat hermitian bundle. We compute in $H^* (X \times S^1;\bQ)$
\[
\pm \ch (\sign_\odd (E \stackrel{\pi}{\to} X;V))\times u\stackrel{\eqref{lem:spectralflow}}{=}
\ch (\sign_\odd (E \stackrel{\pi}{\to} X;V))\times \ch (\sign_\odd (S^1 \times S^1 \stackrel{\proj_1}{\to} S^1;L)) = 
\]
(multiplicativity of the Chern character)
\[
\ch (\sign_\odd (E \stackrel{\pi}{\to} X;V) \times \sign_\odd (S^1 \times S^1 \stackrel{\proj_1}{\to} S^1;L)) \stackrel{\eqref{prop:product-formula-odd-signatureoperator}}{=} 
\]
\[
\frac{(-1)^{m-1}}{2} \ch (\bott (\sign_\even (E \times S^1 \times S^1 \stackrel{\pi \times \proj_1}{\to} X \times S^1;V \boxtimes L))) \stackrel{\eqref{eqn:Chern-Bott-commute}}{=}  
\]
\[
= \frac{(-1)^{m-1}}{2} \ch (\sign_\even (E \times S^1 \times S^1 \stackrel{\pi \times \proj_1}{\to} X \times S^1;V \boxtimes L)) \stackrel{\eqref{thm:index-even-sgnature}}{=}  
\]
\[
-2^{m-1} (\pi \times \proj_1)_! (( \cL (T_v E)\times 1) (\sch (V) \times (1\pm u \times u)))= 
\]
\[
-2^{m-1} (\pi \times \proj_1)_! ( \cL (T_v E)\sch(V) \times  (1\pm u \times u))) \stackrel{\eqref{eqn:crossproduct-gysin}}{=} 
\]
\[
\pm 2^{m-1} \pi_! (\cL(T_v E)\sch(V)) \times u.
\]
As the cross product with $u$ is injective, the claim follows.
\end{proof}

\section{The case of surface groups}\label{sec:surfacegroups}

To deduce Theorem \ref{mainthm:surfacegroup} from Theorem \ref{thm:vanishing}, it remains to show the following. 

\begin{lem}\label{lem:flatbundleonsurface}
For $g \geq 2$, there is a flat hermitian vector bundle $V \to \Sigma_g$ of signature $(1,1)$ with $\scpr{\sch_1 (V),[\Sigma_g]} = 2-2g$.
\end{lem}

\begin{proof}
Recall that $\bP \SU(1,1)= \SU(1,1) / \pm 1$ is the group of orientation-preserving isometries of the Poincar\'e disc model $\bD \subset \bC$ of hyperbolic space; the action is given by M\"obius transformations. The isotropy group of $\SU(1,1)$ at the origin is $\Un(1)$, given by $z \mapsto \mathrm{diag}(z,z^{-1})$. There is a commutative square
\begin{equation}\label{eqn:onbstruction-1}
\xymatrix{
\Spin (2) \ar[d] \ar[r]^-{\cong} & \Un(1) \ar[r]^{\subset} \ar[d] & \SU(1,1)\ar[d]\\
\SO(2) \ar[r]^-{\cong} & \Un(1) /\pm 1 \ar[r]^{\subset} & \bP\SU(1,1)
}
\end{equation}
whose vertical arrows are twofold coverings, and the two inclusions are homotopy equivalences. 

The Riemann uniformization theorem provides us with a homomorphism $\gamma: \Gamma_g \to \bP \SU(1,1)$ such that $\bD / \gamma(\Gamma_g)\cong \Sigma_g$. The composition 
\[
\Sigma_g \stackrel{B \gamma}{\to} B \bP \SU(1,1)^\delta \to B \bP \SU(1,1) \simeq  B\SO(2)
\]
is a classifying map for the tangent bundle $T\Sigma_g$; this is nicely explained in \cite[p.312f]{MilSta}. Since $ \Sigma_g$ is spin, \eqref{eqn:onbstruction-1} shows that the map $\Sigma_g \stackrel{T\Sigma_g}{\to}$ can be lifted to a map to $B\Spin(2)$, and so $B \gamma$ can be lifted to the map $\ell$ in the following diagram:
\begin{equation}\label{eq:obstructiondiagram}
\xymatrix{
& B \SU(1,1)^\delta \ar[r] \ar[d] & B \SU(1,1)\ar[d]  \\
M \ar@{..>}[ur]^{B \rho} \ar@{..>}[urr]^{\ell} \ar[r]^{B \gamma} & B \bP \SU(1,1)^\delta \ar[r]  & B \bP\SU(1,1).\\
}
\end{equation}
The square is homotopy cartesian by general principles: if $G \to H$ is a (surjective) covering of Lie groups with kernel $C$, the homotopy fibre of both $BG \to BH$ and $BG^\delta \to BH^\delta$ is $BC \simeq K(C,1)$. 
Because the square is homotopy cartesian, $\ell$ can further be lifted to a map $M \to B \SU(1,1)^\delta$; since the target is aspherical, the lift must be of the form $B \rho$ for some lift $\rho: \Gamma_g \to \SU(1,1)$ of $\gamma$. 

By construction, the complex line bundle classified by $\Sigma_g \stackrel{B \rho}{\to} B \SU(1,1)^\delta \to B\SU(1,1) \simeq BU(1)$ has Chern number $1-g$. Hence the hermitian vector bundle $V$ given by $B \rho$ has $\scpr{\sch_1 (V),[\Sigma_g]}= 2-2g$. To see where the factor of $2$ originates, note first the commutative diagram
\[
\xymatrix{
\Un(1) \ar[rr]^{z \mapsto \mathrm{diag}(z,z^{-1})} \ar[d]^{z \mapsto \mathrm{diag}(z,z^{-1})} & & \SU(1,1) \ar[d]^{\subset} \\
\Un(1) \times \Un(1) \ar[rr]^{\subset} & & \Un(1,1)
}
\]
whose horizontal maps are homotopy equivalences. In terms of characteristic classes, this means that the pullback of $\sch_1 \in H^2 (\Un(1,1);\bQ)$ along the map $B \Un(1) \to B \Un(1,1)$ induced by the inclusion is the pullback of $\ch_1 \times 1 - 1 \times \ch_1 \in H^2 (\sch_1;\bQ)$ along the map $B(z \mapsto \mathrm{diag}(z,z^{-1})) : B \Un(1) \to B \Un(1) \times B \Un (1)$, which is $2 c_1$. 
\end{proof}

\begin{rem}
In \S $8\frac{2}{7}$ of \cite{Gromov}, Gromov gives a beautiful geometric computation of the index of the even twisted signature operator on $V \to \Sigma_g$; together with the index theorem, this also yields a computation of $\sch_1 (V)$. 
\end{rem}

\section{Nontriviality result}\label{sec:nontriviality}

In this section, we give the proof of Theorem \ref{thm:nontriviality}. The idea is to borrow from the work \cite{GRW17} of Galatius and Randal-Williams bundles of even-dimensional manifolds, and then to take products with fixed manifolds. Let us first record a simple product formula.

\begin{lem}\label{lem:kappa-class-product}
Let $\pi_j:E_j \to X_j$, $f_j: E_j \to BG_j$, $j=0,1$, be two bundles of closed oriented $n_j$-manifolds with maps to $BG_j$. Then for $u_j \in H^{*}(BG_j;\bQ)$, we have
\[
\kappa_{\cL,u_0\times u_1}(E_0 \times E_1,f_0\times f_1) = (-1)^{n_1 |u_0|}\kappa_{\cL,u_0}(E_0,f_0)\times \kappa_{\cL,u_1}(E_1,f_1). 
\]
\end{lem}

\begin{proof}
The first step is to establish the formula 
\begin{equation}\label{eqn:crossproduct-gysin}
(\pi_0 \times \pi_1 )_! (x_0 \times x_1 ) = (-1)^{n_1 (|x_0|-n_0)} (\pi_0)_! (x_0) \times (\pi_1)_! (x_1).
\end{equation}
To see this, recall the general relations $p_! (x \cup p^*y)=p_! (x) \cup y$ and $(p \circ q)_! = p_! \circ q_!$ for the Gysin maps. Furthermore, the Gysin map is natural for pullbacks of bundles, which has the consequence that $(\pi \times \id_Z)_! (x \times 1) = \pi_!(x) \times 1$, whenever $Z$ is a space, $\pi$ an oriented bundle and $1 \in H^0(Z)$ is the unit. With these in mind, compute
\[
(\pi_0 \times \pi_1 )_! (x_0 \times x_1 ) =
(\id_{X_0} \times \pi_1)_! (\pi_0 \times \id_{E_1})_! ( (x_0 \times 1) \cup (\pi_0\times \id_{E_1})^*(1 \times x_1 )) =
\]
\[
=(\id_{X_0} \times \pi_1)_! ((\pi_0 \times \id_{E_1})_! (x_0 \times 1) \cup (1 \times x_1 )) =
(\id_{X_0} \times \pi_1)_! (( (\pi_0)_! (x_0) \times 1 )\cup (1 \times x_1 )) =
\]
\[
=(-1)^{(|x_0|-n_0)|x_1|}(\id_{X_0} \times \pi_1)_! ( (1 \times x_1 ) \cup ((\pi_0)_! (x_0) \times 1 ) ) =
\]
\[
=(-1)^{(|x_0|-n_0)|x_1|}(\id_{X_0} \times \pi_1)_! ( (1 \times x_1 ) \cup (\id_{X_0} \times \pi_1)^* ((\pi_0)_! (x_0) \times 1 ) ) =
\]
\[
= (-1)^{(|x_0|-n_0)|x_1|}(\id_{X_0} \times \pi_1)_!  (1 \times x_1 ) \cup  (\pi_0)_! (x_0) \times 1 )  =
\]
\[
= (-1)^{(|x_0|-n_0)|x_1|}   (1 \times (\pi_1)_!(x_1) ) \cup  (\pi_0)_! (x_0) \times 1 )  =
\]
\[
= (-1)^{(|x_0|-n_0)|x_1|} (-1)^{(|x_1|-n_1)(|x_0|-n_0)} (\pi_0)_! (x_0)  \times (\pi_1)_!(x_1) . 
\]
Hence, using the multiplicativity of the $\cL$-class,
\begin{equation}
\begin{split}
\kappa_{\cL,u_0 \times u_1 }(E_0 \times E_1, f_0 \times f_1)\\
= (\pi_0 \times \pi_1 )_! (\cL (T_v E_0) \times \cL (T_v E_1) \cup (f_0^* (u_0)) \times f_1^* (u_1)) \\=  (\pi_0 \times \pi_1 )_! ((\cL (T_v E_0) \cup f_0^* (u_0)) \times (\cL (T_v E_1) \cup f_1^* (u_1))) \\
= (-1)^{n_1 (|u_0|-n_0)} \kappa_{\cL,u_0} (E_0,f_0) \times   \kappa_{\cL,u_1} (E_1,f_1). 
\end{split}
\end{equation}
\end{proof}
When we split the $\cL$-class into its components, we obtain
\[
\kappa_{\cL_m,u_0\times u_1}(E_0 \times E_1,f_0\times f_1) = \sum_{k+l=m}(-1)^{n_1 (|u_0|-n_0)}\kappa_{\cL_k,u_0}(E_0,f_0)\times \kappa_{\cL_l,u_1}(E_1,f_1). 
\]
We apply this product formula when $\pi_1:E_1 \to X_1$ is the map $c_N:N \to *$ for a single closed manifold of dimension $n_1$. If $4l+|u_1|=n_1$, the previous formula collapses to 
\begin{equation}\label{eqn:productformula-kappanovikov}
\kappa_{\cL_m,u_0\times u_1}(E_0 \times N,f_0\times f_1) = (-1)^{n_1 (|u_0|-n_0)}\kappa_{\cL_{k-l},u_0}(E_0,f_0)\cdot \sign_{u_1}(N,f_1).
\end{equation}

\begin{lem}\label{lem:existence-manifolds-nonzerohighersignautre}
Let $G$ be a discrete group and $v \neq 0\in H^q (BG;\bQ)$. Then for each $k \geq 0$, there is a closed oriented $(4k+q)$-dimensional manifold $N$ and a map $f:N \to BG$ such that $\sign_{v}(N,f) \neq 0$.
\end{lem}

\begin{proof}
This is straightforward from Thom's classical result $\Omega^{\SO}_* (B G)\otimes \bQ \cong H_* (B\SO \times BG;\bQ)$.
\end{proof}

From the work of Galatius and Randal--Williams, we extract the following result. 

\begin{lem}\label{lem:grw-lemma}
Let $n \geq 1$, let $G$ be a discrete group, $v \neq 0 \in H^q (BG;\bQ)$ and $k \in \bN_0$ such that $4k+q\geq 2n$. Then there is a bundle $\pi: E \to X$ of closed oriented $2n$-manifolds, together with a map $h: E \to BG$ such that 
\[
\kappa_{\cL_k,v}(E,h) \neq 0 \in H^{4k+q-2n}(X;\bQ). 
\]
\end{lem}
\begin{proof}
The case $4k+q=2n$ is dealt with in Lemma \ref{lem:existence-manifolds-nonzerohighersignautre}, so let us suppose $4k+q>2n$. There is a finite CW complex $Y$ and a map $g: Y \to BG$ such that $g^* v \neq 0 \in H^q (Y;\bQ)$, by the universal coefficient theorem for rational cohomology. We consider the fibration $\theta: B\SO(2n) \times Y \to B\SO(2n) \to B\mathrm{O}(2n)$, with associated Madsen--Tillmann spectrum $\MT \theta(2n)\simeq \MTSO(2n) \wedge Y_+$. The idea is to apply \cite[Theorem 7.3]{GRW17} to $\theta$. A $\theta$-structure on a manifold is an orientation together with a map to $Y$. The sphere $S^{2n-1}$ has a $\theta$-structure, given by the constant map to $Y$.

Let $\mathscr{N}_n^\theta (S^{2n-1})$ denote the space of all $\theta$-nullbordisms of $S^{2n-1}$ with $n$-connected structure map to $B\SO(2n) \times Y$. 
The space $\mathscr{N}_n^\theta (S^{2n-1})$ classifies bundles of oriented $2n$-manifolds with boundary $S^{2n-1}$, together with maps to $Y$ on the total space; these maps are constant on the boundary, and on each fibre the induced map to $B\SO(2n) \times Y$ is $n$-connected. One can close the boundary by adding a fixed copy of $D^{2n}$ (with the constant map to $Y$). 

Theorem 7.3 of \cite{GRW17} states that there is a homology isomorphism 
\[
\alpha: \hocolim_{g \to \infty} (\mathscr{N}_n^\theta (S^{2n-1}) \to \mathscr{N}_n^\theta (S^{2n-1}) \to \ldots ) \to \Omega^\infty (\MTSO(2n) \wedge Y_+),
\]
where the stabilization map is by adding $(S^{2n-1} \times [0,1]) \sharp (S^n \times S^n)$ (with the constant map to $Y$).

Under the Thom isomorphism $H^{*+2n} (B\SO(2n) \times Y;\bQ) \cong H^* (\MTSO(2n) \wedge Y_+;\bQ)$, the class $\cL_k \times v$ corresponds to a class in $H^{4k+q-2n}(\MTSO(2n)\wedge Y_+;\bQ)$, which produces a class in $H^{4k+q-2n}(\Omega^\infty \MTSO(2n);\bQ)$ which is nontrivial as $v \neq 0$, $\cL_k \neq 0$ and $4k+q-2n \geq 0$. Hence this class becomes nonzero when pulled back to $\mathscr{N}_n^\theta (S^{2n-1})$. This class agrees with the class $\kappa_{\cL_k,v}$, evaluated on the bundle with closed off boundary. This finishes the proof.
\end{proof}

\begin{proof}[The proof of Theorem \ref{thm:nontriviality}] is now straightforward: 
Lemma \ref{lem:grw-lemma} provides a bundle $\pi:E \to X$ with oriented $2m_0$-dimensional fibres and $f:E \to BG$ such that $\kappa_{\cL_{m_0},v}(E,f) \neq 0$, and there is a closed $2m_1+1$-dimensional $N$ and $g:N \to BH$ with $\sign_{w} (N,g) \neq 0$. The bundle $E \times N \to X$ with the map $f \times g: E \times N \to B(G \times H)$ does the job by Lemma \ref{lem:kappa-class-product}. 
\end{proof}

\bibliographystyle{plain}
\bibliography{highersignatures}

\end{document}